\documentclass[12pt]{article}
\usepackage[usenames]{color}
\usepackage{pst-plot}
\usepackage{url}
\usepackage{subfigure}
\usepackage{amsfonts,mathrsfs}
\usepackage{amssymb,amsmath,amsthm}
\usepackage{verbatim}
\usepackage{acronym}
\usepackage{epsfig}
\usepackage{graphicx}
\usepackage{enumerate}

\def\fskip#1{}

\newtheorem{theorem}{Theorem}

\newtheorem{corollary}{Corollary}

\newtheorem{definition}{Definition}

\newtheorem{lemma}{Lemma}

\newcommand{\Rm}{\mathbb{R}^m}
\newcommand{\M}{\mathscr{M}}
\newcommand{\Einf}{\mathcal{E}^{\infty}}

\def\1{{\bf 1}}
\def\E{\mathcal{E}}

\def\e{{\bf e}}

\def\e{\epsilon}

\newcommand{\al}{\alpha}

\def\R{\mathbb{R}}

\def\Z{\mathbb{Z}}

\newcommand{\Ginf}{G^{\infty}}
\newcommand{\Ac}{\{A(k)\}}
\newcommand{\Atc}{\{\tilde{A}(k)\}}
\newcommand{\At}{\tilde{A}}
\newcommand{\Bc}{\{B(k)\}}

\newcommand{\Ic}{\{I\}}

\newcommand{\Qi}{P^{(\xi)}}

\newcommand{\Pc}{\{P(k)\}}
\newcommand{\Ps}{\mathscr{P}}

\newcommand{\perm}{\mathscr{P}}

\newcommand{\Sc}{\{S(k)\}}
\newcommand{\Tc}{\{T(k)\}}
\newcommand{\xc}{\{x(k)\}}
\newcommand{\yc}{\{y(k)\}}

\title{On Backward Product of Stochastic Matrices}
\author{Behrouz Touri and Angelia Nedi\'c\thanks{Department of Industrial
and Enterprise Systems Engineering, University of
Illinois, Urbana, IL 61801,
Email: \{touri1,angelia\} @illinois.edu. This research
is supported by the National Science Foundation under
CAREER grant CMMI 07-42538.}
}
\begin{document}
\maketitle
\begin{abstract}
We study the ergodicity of backward product of stochastic and doubly stochastic
matrices by introducing the concept of absolute infinite flow property.
We show that this property is necessary for ergodicity of any chain of
stochastic matrices, by defining and exploring the properties of a
rotational transformation for a stochastic chain.
Then, we establish that the absolute infinite flow property is equivalent
to ergodicity for doubly stochastic chains. Furthermore, we develop a rate of
convergence result for ergodic doubly stochastic chains. We also
investigate the limiting behavior of a doubly stochastic chain and
show that the product of doubly stochastic matrices is convergent up to
a permutation sequence. Finally, we apply the results to
provide a necessary and sufficient condition for the absolute asymptotic
stability of a discrete linear inclusion driven by doubly stochastic matrices.
\end{abstract}

\section{Introduction}\label{sec:introduction}

The study of forward product of an inhomogeneous chain of stochastic
matrices is closely related to the limiting behavior, especially
ergodicity, of inhomogeneous Markov chains. The earliest study on the
forward product of inhomogeneous chains of stochastic matrices is the work of
Hajnal in~\cite{hajnal}. Motivated by a homogeneous Markov chain, Hajnal
formulated the concepts of ergodicity in weak and strong senses for
inhomogeneous Markov chains and developed some sufficient conditions for both
weak and strong ergodicity of such chains. Using the properties of
scrambling matrices that were introduced in~\cite{hajnal},
Wolfowitz~\cite{wolfowitz} gave a condition under which all the chains driven
from a finite set of stochastic matrices are strongly ergodic. In his elegant
work \cite{Shen}, Shen gave geometric interpretations and provided some
generalizations of the results in~\cite{hajnal} by considering vector
norms other than $\|\cdot\|_\infty$, which was originally used in~\cite{hajnal}
to measure the scrambleness of a matrix.

The study of backward product of row-stochastic matrices, however,
was motivated by different applications all of which were in search of a form
of a consensus between a set of processors, individuals, or agents.
DeGroot~\cite{DeGroot} studied such a product (for a homogeneous chain)
as a tool for reaching consensus on a distribution of a certain unknown
parameter among a set of agents. Later, Chatterjee and Seneta~\cite{SenetaCons}
provided a theoretical framework for reaching consensus by studying
the backward product of an inhomogeneous chain of stochastic matrices.
Motivated by the theory of inhomogeneous Markov chains, they defined
the concepts of weak and strong ergodicity in this context and showed that
those two properties are equivalent. Furthermore, they developed the theory of
coefficients for ergodicity. Motivated by some distributed computational
problems, in \cite{Tsitsiklis86}, Tsitsiklis and Bertsekas studied such
a product from the dynamical system point of view. In fact, they considered
a dynamics that enables an exogenous input as well as delays in the system.
Through the study of such dynamics, they gave a more practical conditions for
a chain to ensure the consensus. The work in~\cite{Tsitsiklis86} had a great
impact on the subsequent studies of distributed estimation and
control problems.

The common ground in the study of both forward and backward products of
stochastic matrices are the chains of doubly stochastic matrices.
By transposing the matrices in such a chain, forward product of
matrices can be transformed into backward product of the transposes of the
matrices. Therefore, in the case of doubly stochastic matrices,
any property of backward products translates to the same property for
forward products.
However, since the transposition of a row-stochastic matrix is not
necessarily a row-stochastic matrix, the behavior of the
forward and backward products of stochastic matrices can be quite different
in general.

Here, we study the backward product of a chain of stochastic matrices
in general, and then focus on doubly stochastic matrices in particular.
We start by introducing a concept of {\it absolute infinite flow property}
as a refinement of infinite flow property, proposed in our earlier work
in~\cite{ErgodicityPaper}. We demonstrate that
the absolute infinite flow property
is more stringent than the infinite flow property, by showing that
a sequence of permutation matrices cannot have the
absolute infinite flow property, while it may have the infinite flow property.
We introduce a concept of
{\it rotational transformation}, which plays a central role in our exploration
of the relations between the absolute infinite flow property and ergodicity of
a stochastic chain.
In fact, using the properties of the rotational transformation, we show that
the {\it absolute infinite flow property is necessary for ergodicity of any
stochastic chain}.
However, even though this property requires a rich structure for
a stochastic chain, it is still not sufficient for the ergodicity, which we
illustrate on an example. In pursue of chains for which this property is
sufficient for ergodicity, we come to identify a class of
{\it decomposable chains},
which have a particularly simple test for the absolute infinite flow property.
Upon exploring decomposable chains, we show that doubly stochastic chains
are decomposable by using Birkhoff-von Neumann theorem.
Finally, we show that the {\it ergodicity and absolute infinite flow property
are equivalent for doubly stochastic chains} and provide
a convergence rate result for ergodic doubly stochastic chains
by establishing a result on limiting behavior of doubly stochastic chain
that need not be ergodic. We conclude by considering a discrete linear
inclusion system driven by a subset of stochastic or doubly stochastic
matrices. We provide a necessary and sufficient conditions for
the absolute asymptotic stability of such discrete linear inclusions.

This work is a continuation of our earlier work
in~\cite{ErgodicityPaper,TouriNedich:Approx} where
we have studied random backward products for
independent, not necessarily identically distributed, random matrix processes
with row-stochastic matrices. In~\cite{ErgodicityPaper},
we have introduced \textit{infinite flow property} and
shown that this property is necessary for ergodicity of any chain
of row-stochastic matrices. Moreover, we have shown that this property
is also sufficient for ergodicity of such chains
under some additional conditions on the matrices.
In~\cite{TouriNedich:Approx}, we have extended some of our
results from~\cite{ErgodicityPaper} to a larger class of row-stochastic
matrices by using some approximations of chains. The setting
in~\cite{ErgodicityPaper,TouriNedich:Approx} has been within
random matrix processes with row-stochastic matrices.
Unlike our work in~\cite{ErgodicityPaper,TouriNedich:Approx},
the work in this present paper is focused on deterministic sequences of
row-stochastic matrices. Furthermore,
the new concept of the absolute infinite flow property is demonstrated to be
significantly stronger than the infinite flow property. Finally,
our line of analysis in this paper relies on the development of some new
concepts, such as {\it rotational transformation and decomposable chains}.
Our result on convergence rate for ergodic doubly stochastic chains extends
the rate result in~\cite{NedichOlshevskyOzdaglar:09} to a larger class of
doubly stochastic chains, such as those that may not have uniformly bounded
positive entries or uniformly bounded diagonal entries.

The main new results of this work include:
(1) Formulation and development of
\textit{absolute infinite flow property}. We show that this property
is necessary for ergodicity of any stochastic chain, despite the fact
that it is much stronger than infinite flow property.
(2) Introduction and exploration of \textit{rotational transformation}
of a stochastic chain with respect to a permutation chain. We establish that
rotational transformation preserves several properties of a chain,
such as ergodicity and absolute infinite flow property.
(3) Establishment of equivalence between ergodicity and absolute infinite flow
property for doubly stochastic chains. We accomplish this
through the use of the Birkhoff-von Neumann decomposition of doubly stochastic
matrices and properties of rotational transformation of a stochastic chain.
(4) Development of a rate of convergence for ergodic doubly stochastic chains.
(5) Establishment of the limiting behavior for doubly stochastic chains.
We prove that a product of doubly stochastic matrices is convergent up
to a permutation sequence. (6) Development of necessary and sufficient conditions for the absolute asymptotic stability of a discrete linear inclusion system driven by doubly stochastic matrices.

The structure of this paper is as follows: in Section~\ref{subsec:notation},
we introduce the notation that we use throughout the paper.
In Section~\ref{sec:ergodic}, we discuss the concept of ergodicity and
introduce the absolute infinite flow property. In Section~\ref{sec:absolute},
we introduce and investigate the concept of rotational transformation.
Using the properties of this transformation,
we show that this property is necessary for ergodicity of any stochastic chain.
Then, in Section~\ref{sec:decomposable} we introduce decomposable chains and
study their properties. In Section~\ref{sec:doublystoch},
we study the product of doubly stochastic matrices by exploring ergodicity,
rate of convergence, and the limiting behavior in the absence of ergodicity.
In Section~\ref{sec:switch}, we study the absolute asymptotic stability of
a discrete inclusion system driven by doubly stochastic matrices.
We summarize the development in Section~\ref{sec:conclusion}.

\subsection{Notation and Basic Terminology}\label{subsec:notation}
We view all vectors as columns. For a vector $x$, we write $x_i$
to denote its $i$th entry, and we write $x\ge0$ ($x>0$) to denote that
all its entries are nonnegative (positive). We use $x^T$ to denote
the transpose of a vector $x$. We write $\|x\|$ to denote the standard
Euclidean vector norm i.e., $\|x\|=\sqrt{\sum_{i}x_i^2}$.
We use $e_i$ to denote the vector with the $i$th entry equal to~1 and
all other entries equal to~$0$, and we write
$e$ for the vector with all entries equal to~$1$.
We say that a vector $v\in\Rm$ is stochastic if $v_i\ge 0$ for all $i$
and $\sum_{i=1}^{m}v_i=1$. For a given set $C$ and
a subset $S$ of $C$, we write $S\subset C$ to denote
that $S$ is a proper subset of $C$. A set $S\subset C$ such that
$S\ne\emptyset$ is referred to as a \textit{nontrivial} subset of $C$.
We write $[m]$ to denote the integer set $\{1,\ldots,m\}$.
For a set $S\subset[m]$, we let $|S|$ be the cardinality of the set $S$
and $\bar S$ be the complement of $S$ with respect to $[m]$, i.e.,
$\bar S=\{i\in[m]\mid i\notin S\}$.

We denote the identity matrix by $I$. For a matrix $A$, we use  $A_{ij}$ to
denote its $(i,j)$th entry, $A_i$ and $A^j$ to denote its $i$th row and
$j$th column vectors, respectively, and $A^T$ to denote its transpose.
We write $\|A\|$ for the matrix norm induced by the Euclidean vector norm.
A matrix $A$ is {\it row-stochastic} when all of its rows are stochastic
vectors. Since we deal exclusively with
row-stochastic matrices, we refer to such matrices simply as {\it stochastic}.
A matrix $A$ is {\it doubly stochastic} when both $A$ and $A^T$ are stochastic.
We say that a chain of matrices $\Ac$ is a stochastic chain if $A(k)$ is
a stochastic matrix for all $k\geq 0$. Similarly, we say that $\Ac$ is
a doubly stochastic chain if $A(k)$ is a doubly stochastic matrix for all
$k\geq 0$. We may refer to a stochastic chain or a doubly stochastic chain
simply as a chain when there is no ambiguity on the nature of
the underlying chain.

For an $m\times m$ matrix $A$, we use $\sum_{i<j}A_{ij}$ to denote
the summation of the entries $A_{ij}$ over all $i,j\in [m]$ with $i<j$.
Given a nonempty index set $S\subseteq[m]$ and a matrix $A$, we write $A_S$
to denote the following summation:
\[A_S=\sum_{i\in S,j\in\bar{S}}
A_{ij}+\sum_{i\in \bar{S},j\in S}A_{ij}.\]
Note that $A_S$ satisfies $A_S=\sum_{i\in S,j\in\bar{S}}(A_{ij}+A_{ji})$.

An $m\times m$ matrix $P$ is a permutation matrix if it contains exactly one
entry equal to $1$ in each row and each column. Given a permutation matrix $P$,
we use $P(S)$ to denote the image of an index set $S\subset[m]$
under the permutation $P$;
specifically $P(S)=\{i\in[m]\mid P_{ij}=1\mbox{ for some $j\in S$}\}$.
We note that a set $S\subset [m]$ and its image $P(S)$ under a permutation $P$
have the same cardinality, i.e., $|S|=|P(S)|$. Furthermore, for any permutation
matrix $P$ and any nonempty index set $S\subset [m]$,
the following relation holds:
\[\sum_{i\in P(S)}e_i = P\sum_{j\in S}e_j.\]
We denote the set of $m\times m$ permutation matrices by $\perm_m$. Since there
are $m!$ permutation matrices of size $m$, we may assume that the set of
permutation matrices is indexed, i.e., $\perm_m=\{\Qi\mid 1\leq \xi\leq m!\}$.
Also, we say that $\Pc$ is a permutation sequence if $P(k)\in \perm_m$ for all
$k\geq 0$. The sequence $\{I\}$ is the permutation sequence
$\{P(k)\}$ with $P(k)=I$ for all $k$, and it is
referred to as the \textit{trivial permutation sequence}.

We view a directed graph $G$ as an ordered set $(V,E)$ where $V$ is a finite
set of vertices (the vertex set)
and $E\subseteq V\times V$ is the edge set of $G$.
For two vertices $u,v\in V$, we often use $u\to v$ to denote the edge
$(u,v)\in E$. A cycle in $G=(V,E)$ is a sequence of vertices
$u_0,u_1,\ldots,u_{r-1},u_r=u_0$ such that $u_\tau\to u_{\tau+1}$ for all
$\tau\in \{0,\ldots,r-1\}$. We denote such a cycle by
$u_0\to u_1\to\cdots u_{r-1}\to u_r=u_0$. Throughout this paper,
we assume that the vertices of a cycle are distinct. Note that a loop
$(v,v)\in E$ is also a cycle. We say that a directed graph $G$ is cycle-free
when $G$ does not contain a cycle. We say that $u_0\to u_1\to u_2\to \cdots$
is a walk in $G$ if $u_i\to u_{i+1}$ for all $i\geq 0$.

\section{Ergodicity and Absolute Infinite Flow}\label{sec:ergodic}
In this section, we discuss the concepts of ergodicity and infinite flow
property, and we introduce a more restrictive property than infinite flow,
which will be a central concept in our later development.
\subsection{Ergodic Chain}\label{sec:ergchain}
Here, we define the ergodicity for a backward product of a chain $\Ac$ of
$m\times m$ stochastic matrices $A(k)$.
For $k>s\geq 0$, let $A(k:s)=A(k-1)A(k-2)\cdots A(s)$. We use the following
definition for ergodicity.
\begin{definition}\label{def:ergodicity}
We say that a chain $\Ac$ is \textit{ergodic} if
\[\lim_{k\rightarrow\infty}A(k:s)=ev^T(s)\qquad\mbox{for all $s\geq 0$},\]
where $v(s)$ is a stochastic vector for all $s\geq 0$.
\end{definition}

In other words, a chain $\Ac$ is ergodic if the infinite backward product
$\cdots A(s+2)A(s+1)A(s)$ converges to a matrix with identical rows,
which must be stochastic vectors since the chain $\Ac$ is stochastic.

To give an intuitive understanding of the backward product $A(k:s)$ of a
stochastic chain $\Ac$, consider the following simple model for opinion
dynamics for a set $[m]=\{1,\ldots,m\}$ of $m$ agents. Suppose that
at time $s$, each of the agents has a belief about a certain issue that can be
represented by a scalar $x_i(s)\in\R$. Now, suppose that agents' beliefs evolve
in time by the following dynamics: at any time $k\geq s$, the agents meet and
discuss their beliefs about the issue and, at time $k+1$, they
update their beliefs by taking convex combinations of the agents' beliefs
at the prior time $k$, i.e.,
\begin{align}\label{eqn:hypopiniondynamics}
x_i(k+1)=\sum_{j=1}^mA_{ij}(k)x_j(k)\qquad \mbox{for all $i\in[m]$},
\end{align}
with $\sum_{j=1}^mA_{ij}(k)=1$ for all $i\in[m]$, i.e., $A(k)$ is stochastic.
In this opinion dynamics model, we are interested in properties of the chain
$\Ac$ that will ensure the $m$ agents reach a consensus on their beliefs
for any starting time $s\geq 0$ and any initial belief profile $x(s)\in \Rm$.
Formally, we want to determine some restrictions on the chain which will
guarantee that $\lim_{k\rightarrow\infty}\left(x_i(k)-x_j(k)\right)=0$ for any
$i,j\in[m]$ and any $s\geq 0$. As proven in~\cite{SenetaCons}, this limiting
behavior is equivalent to ergodicity of $\Ac$ in the sense of
Definition~\ref{def:ergodicity}.

One can visualize the dynamics in Eq.~\eqref{eqn:hypopiniondynamics} using
the \textit{trellis graph} associated with a given stochastic chain.
The trellis graph of a stochastic chain $\Ac$ is an infinite directed
weighted graph
$G=(V,E,\Ac)$, with the vertex set $V$ equal to the infinite grid
$[m]\times \mathbb{Z}^+$ and the edge set
\begin{align}\label{eqn:trellisedges}
E=\{\left((j,k),(i,k+1)\right)\mid j,i\in[m],\ k\geq 0 \}.
\end{align}
In other words, we consider a copy of the set $[m]$ for any time $k\geq 0$ and
we stack these copies over time, thus generating the infinite vertex set
$V=\{(i,k) \mid i\in[m],\ k\ge0\}$.
We then place a link from each $j\in[m]$ at time $k$ to every $i\in[m]$
at time $k+1$, i.e., a link from each vertex $(j,k)\in V$ to each vertex
$(i,k+1)\in V$. Finally, we assign the weight $A_{ij}(k)$ to the link
$((j,k),(i,k+1))$. Now, consider the whole graph as an information tunnel
through which the information flows: we inject a scalar $x_i(0)$ at each
vertex $(i,0)$ of the graph. Then, from this point on, at each time
$k\geq 0$, the information is transfered from time $k$ to time $k+1$
through each edge of the graph that acts as a communication link.
Each link attenuates the in-vertex's value with its weight, while each vertex
sums the information received through the incoming links. One can observe that
the resulting information evolution is the same as the dynamics given in
Eq.~\eqref{eqn:hypopiniondynamics}. As an example, consider the  $2\times 2$
static chain $\Ac$ with $A(k)$ defined by:
\begin{align}\label{eqn:trellisexample}
A(k)=\left[
\begin{array}{cc}
\frac{1}{4}&\frac{3}{4}\\
\frac{3}{4}&\frac{1}{4}
\end{array}
\right]\qquad \mbox{for $k\geq 0$}.
\end{align}
The trellis graph of this chain and the resulting dynamics is depicted in
Figure~\ref{fig:trellisexample}.

\begin{figure}
\centering
\includegraphics[width=0.8\linewidth]{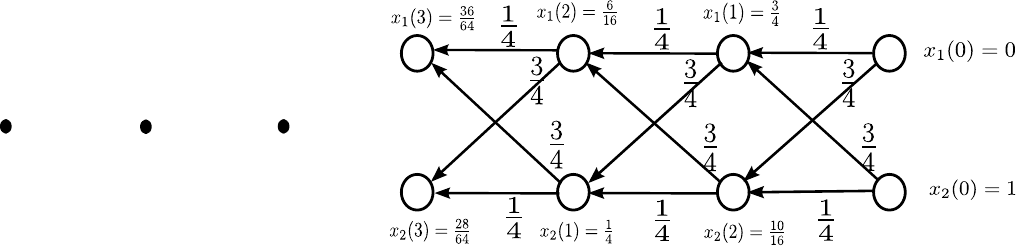}
\caption{\label{fig:trellisexample}
The trellis graph of the $2\times 2$ chain $\Ac$ with weights
$A(k)$ given in Eq.~\eqref{eqn:trellisexample}.}
\end{figure}

\subsection{Infinite Flow and Absolute Infinite Flow}\label{sec:infflow}
Here, we discuss the infinite flow property, as introduced
in~\cite{ErgodicityPaper}, and a more restrictive version of this
property which turned out to be necessary for ergodicity of stochastic chains,
as shown later on in Section~\ref{sec:rotation}.

We start by recalling the definition of the infinite flow property
from~\cite{ErgodicityPaper}.
\begin{definition}\label{def:infiniteflow}
A stochastic chain $\Ac$ has the infinite flow property if
\[\sum_{k=0}^{\infty}A_S(k)=\infty\qquad
\hbox{for all nonempty $S\subset[m]$},\]
where $A_S(k)=\sum_{i\in S,j\in\bar{S}}\left(A_{ij}(k)+A_{ji}(k)\right)$.
\end{definition}

Graphically, the infinite flow property requires that in the trellis graph of
a given model $\Ac$, the weights on the edges between $S\times \Z^+$ and
$\bar{S}\times \Z^+$ sum up to infinity.

To illustrate the infinite flow property, consider the opinion dynamics in
Eq.~\eqref{eqn:hypopiniondynamics}. One can interpret $A_{ij}(k)$ as the
\textit{credit} that agent $i$ gives to agent $j$'s opinion at time $k$.
Therefore, the sum $\sum_{i\in S,j\in\bar{S}}A_{ij}(k)$ can be interpreted as
the credit that the agents' group $S\subset[m]$ gives to the opinions of
the agents that are outside of $S$ (the agents in $\bar{S}$) at time $k$.
Similarly, $\sum_{i\in \bar{S},j\in S}A_{ij}(k)$ is the credit that the agents'
group $\bar{S}$ gives to the agents in group $S$. The intuition behind
the concept of infinite flow property is that without having infinite
accumulated credit between agents' groups $S$ and $\bar{S}$,
we cannot have an agreement among the agents for any starting time $s$
of the opinion dynamic and for any opinion profile $x(s)$ of the agents.
In other words, the infinite flow property ensures that there is enough
information flow between the agents as time passes by.

We have established in~\cite{ErgodicityPaper} that the infinite flow property
is necessary for ergodicity of a stochastic chain, as restated below
for later use.

\begin{theorem}\label{thrm:necessaryergodicity}\cite{ErgodicityPaper}
The infinite flow property is a necessary condition for ergodicity of
any stochastic chain.
\end{theorem}

Although the infinite flow property is necessary for ergodicity,
this property alone is not \textit{strong enough} to  disqualify some
stochastic chains from being ergodic, such as chains of permutation matrices.
As a concrete example consider a static chain $\{A(k)\}$ of
permutation matrices $A(k)$ given by
\begin{align}\label{eqn:examplepermutation}
A(k)=
\left[
\begin{array}{cc} 0&1\\ 1&0 \end{array}\right]\qquad\mbox{for $k\geq 0$}.
\end{align}
As a remedy for this situation, in~\cite{ErgodicityPaper} we
have imposed some additional conditions, including some feedback-type
conditions, on the matrices $A(k)$ that eliminate permutation matrices.
Here, we take a different approach.

Specifically, we will require a stronger infinite flow property
by letting the set $S$ in Definition~\ref{def:infiniteflow} vary with time.
In order to do so, we will consider sequences $\Sc$ of index sets
$S(k)\subset[m]$ with a form of regularity in the sense that
the sets $S(k)$ have the same cardinality for all $k$.
In what follows, we will reserve notation $\Sc$ for the sequences
of index sets $S(k)\subset[m]$. Furthermore,
for easier exposition, we define the notion of regularity
for $\Sc$ as follows.

\begin{definition}\label{def:regular}
A sequence $\Sc$ is regular if the sets $S(k)$ have the same (nonzero)
cardinality, i.e., $|S(k)|=|S(0)|$ for all $k\ge0$ and $|S(0)|\ne0$.
\end{definition}

The nonzero cardinality requirement in Definition~\ref{def:regular} is imposed
only to exclude the trivial sequence $\Sc$ consisting of empty sets.

Graphically, a regular sequence $\Sc$ corresponds to the subset
$\{(i,k)\mid i\in S(k),\ k\geq 0\}$ of vertices in the trellis graph
associated with a given chain. As an illustration, let us
revisit the $2\times 2$ chain given in Eq.~\eqref{eqn:examplepermutation}.
Consider the regular $\Sc$ defined by
\[S(k)=\left\{\begin{array}{cc}
\{1\}&\mbox{if $k$ is even},\\
\{2\}&\mbox{if $k$ is odd}.
\end{array}\right.\]
The vertex set $\{(i,k)\mid i\in S(k),\ k\geq 0\}$ associated with
$\Sc$ is shown in Figure~\ref{fig:infiniteflow}.

\begin{figure}
\begin{center}
\includegraphics[width=0.7\linewidth]{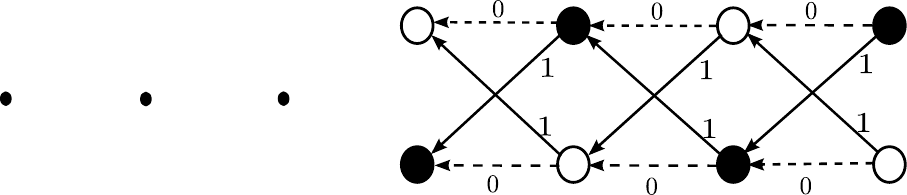}
\end{center}
\caption{\label{fig:infiniteflow} The trellis graph of the permutation chain
in Eq.~\eqref{eqn:examplepermutation}. For the regular sequence
$\Sc$, with $S(k)=\{1\}$ if $k$ is even and $S(k)=\{2\}$ otherwise,
the vertex set $\{(i,k)\mid i\in S(k),\ k\geq 0\}$ is marked by black vertices.
The flow $F(\Ac;\Sc)$, as defined in~\eqref{eqn:scflow},
corresponds to the summation of the weights on the dashed edges.}
\end{figure}

Now, let us consider a chain $\Ac$ of stochastic matrices $A(k)$.
Let $\Sc$ be any regular sequence. At any time $k$,
we define the flow associated with the entries of the matrix $A(k)$
across the index sets $S(k+1)$ and $S(k)$ as follows:
\begin{align}\label{eqn:flow}
A_{{S}(k+1),S(k)}(k)=\sum_{i\in S(k+1),j\in\bar{S}(k)}
A_{ij}(k)+\sum_{i\in \bar{S}(k+1),j\in S(k)}A_{ij}(k)\qquad\hbox{for $k\ge0$}.
\end{align}
The flow $A_{{S}(k+1),S(k)}(k)$ could be viewed as an instantaneous flow at
time $k$ induced by the corresponding elements in the matrix chain and the
index set sequence. Accordingly, we define the total flow of a chain $\Ac$ over
$\Sc$, as follows:
\begin{align}\label{eqn:scflow}
F(\Ac;\Sc)=\sum_{k=0}^{\infty}A_{{S}(k+1),S(k)}(k).
\end{align}
We are now ready to extend
the definition of the infinite flow property to time varying index sets $S(k)$.

\begin{definition}\label{def:absoluteinfflow}
A stochastic chain $\Ac$ has the absolute infinite flow property if
\[F(\Ac;\Sc)=\infty\quad
\hbox{for every regular sequence $\Sc$.}\]
\end{definition}

Note that the absolute infinite flow property of
Definition~\ref{def:absoluteinfflow} is more restrictive than
the infinite flow property of Definition~\ref{def:infiniteflow}.
In particular, we can see this by letting the set sequence $\Sc$
be static, i.e., $S(k)=S$ for all $k$ and some nonempty $S\subset[m]$.
In this case, the flow $A_{{S}(k+1),S(k)}(k)$ across the index sets
$S(k+1)$ and $S(k)$ as defined in Eq.~\eqref{eqn:flow} reduces to
$A_S(k)$, while the flow $F(\Ac;\Sc)$ in Eq.~\eqref{eqn:scflow} reduces to
$\sum_{k=0}^{\infty}A_S(k)$. This brings us to the quantities
that define the infinite flow property (Definition~\ref{def:infiniteflow}).
Thus, the infinite flow property requires that
the flow across a trivial regular sequence $\{S\}$ is infinite
for all nonempty $S\subset[m]$, which is evidently less restrictive
requirement than that of Definition~\ref{def:absoluteinfflow}.
In light of this, we see that if a stochastic chain $\Ac$ has the absolute
infinite property, then it has the infinite flow property.

The distinction between the absolute infinite flow property and
the infinite flow property is actually much deeper. Recall our example of
the chain $\Ac$ in Eq.~\eqref{eqn:examplepermutation}, which demonstrated that
a permutation chain may posses the infinite flow property.
Now, for this chain, consider the regular
sequence $\Sc$ with $S(2k)=\{1\}$ and $S(2k+1)=\{2\}$ for $k\geq 0$.
The trellis graph associated with $\Ac$ is shown in
Figure~\ref{fig:infiniteflow}, where $\Sc$ is depicted by black vertices.
The flow $F(\Ac;\Sc)$ corresponds to the summation of the weights on the dashed
edges in~Figure~\ref{fig:infiniteflow}, which is equal to zero in this case.
Thus, the chain $\Ac$ does not have the absolute flow property.

In fact, while some chains of permutation matrices may have the infinite flow
property, it turns out that no chain of permutation matrices has
the absolute infinite flow property. In other words, the absolute infinite flow
property is strong enough to filter out the chains of permutation matrices
in our search for necessary and sufficient conditions for ergodicity which
is a significant distinction between
the absolute infinite flow property and the infinite flow property.
To formally establish this, we turn our attention to an intimate
connection between a regular index sequence and a permutation sequence that
can be associated with the index sequence.

Specifically, an important feature of a regular sequence $\Sc$ is that it can
be obtained as the image of the initial set $S(0)$ under a certain permutation
sequence $\Pc$. This can be seen by noting that
we can always find a one-to-one matching
between the indices in $S(k)$ and $S(k+1)$ for all $k$,
since $|S(k)|=|S(k+1)|$.
As a result, any set $S(k)$ is an image of a finitely many permutations of
the initial set
$S(0)$; formally $S(k)=P(k-1)\cdots P(1)P(0)(S(0))$.
Therefore, we will refer to the set $S(k)$ {\it as the image of the set
$S(0)$ under $\Pc$ at time $k$}. Also, we will refer to the sequence $\Sc$
as {\it the trajectory of the set $S(0)$ under $\Pc$}.

As a first consequence of the definition,
we show that no chain of permutation matrices has
the absolute infinite property.

\begin{lemma}\label{lemma:noperm}
For any permutation chain $\Pc$, there exists a regular sequence $\Sc$
for which \[F(\Pc;\Sc)=0.\]
\end{lemma}
The result follows by letting $\Sc$ be the trajectory of
a nonempty set $S(0)\subset[m]$ under the permutation $\Pc$.

\section{Necessity of Absolute Infinite Flow for Ergodicity}
\label{sec:absolute}
As discussed in Theorem~\ref{thrm:necessaryergodicity}, the infinite
flow property is necessary for ergodicity of a stochastic chain.
In this section, we show that the absolute infinite flow
property is actually necessary for ergodicity of a stochastic chain despite
the fact that this property is much more restrictive than the infinite flow
property. We do this by considering a stochastic chain
$\Ac$ and a related chain, say $\Bc$, such that
the flow of $\Ac$ over a trajectory translates to a flow
of $\Bc$ over an appropriately defined trajectory. The technique that
we use for defining the chain $\Bc$ related to a given chain $\Ac$ is
developed in the following section. Then, in Section~\ref{sec:necessary},
we prove the necessity of the absolute infinite flow for ergodicity.
\subsection{Rotational Transformation}\label{sec:rotation}
Rotational transformation is a process that takes a chain and produces
another chain through the use of a permutation sequence $\Pc$.
Specifically, we have the following definition of the rotational transformation
with respect to a permutation chain.

\begin{definition}\label{def:rotation}
Given a permutation chain $\Pc$, the rotational transformation of an arbitrary
chain $\Ac$ with respect to $\Pc$ is the chain $\Bc$ given by
\[B(k)=P^T(k+1:0)A(k)P(k:0)\qquad\hbox{for $k\geq 0$},\]
where $P(0:0)=I$. We say that $\Bc$ is the rotational transformation
of $\Ac$ by $\Pc$.
\end{definition}

The rotational transformation has some interesting properties
for stochastic chains which we discuss in the following lemma.
These properties play a key role in the subsequent development, while they may
also be of interest in their own right.

\begin{lemma}\label{lemma:rotationproperties}
Let $\Ac$ be an arbitrary stochastic chain and $\Pc$ be an arbitrary
permutation chain. Let $\Bc$ be the rotational transformation of $\Ac$ by
$\Pc$. Then, the following statements are valid:
\begin{enumerate}[(a)]
\item \label{rot:trans}
The chain $\Bc$ is stochastic. Furthermore,
\[B(k:s)=P^T(k:0)A(k:s)P(s:0)\qquad\mbox{for any $k> s\ge0$},\]
where $P(0:0)=I$.
\item \label{rot:equi}
The chain $\Ac$ is ergodic if and only if the chain $\Bc$ is ergodic.
\item \label{rot:absflow} For any regular sequence $\Sc$ for $\Ac$,
there exists another regular sequence $\{T(k)\}$ for $\Bc$
such that $B_{T(k+1)T(k)}(k)=A_{S(k+1)S(k)}(k)$. Also, for any
regular sequence $\{T(k)\}$ for $\Bc$, there exists a regular sequence
$\Sc$ for $\Ac$ with $A_{S(k+1)S(k)}(k)=B_{T(k+1)T(k)}(k)$. In particular,
$F(\Ac;\Sc)=F(\Bc;\{T(k)\})$ and hence,
$\Ac$ has the absolute infinite flow property if and only if
$\Bc$ has the absolute infinite flow property.
\item \label{rot:flow}
For any $S\subset[m]$ and $k\geq 0$, we have
$A_{S(k+1),S(k)}(k)=B_{S}(k)$, where $S(k)$ is the image of
$S$ under $\Pc$ at time $k$, i.e., $S(k)=P(k:0)(S)$.
\end{enumerate}
\end{lemma}

\begin{proof}
\noindent (a)\quad
By the definition of $B(k)$, we have
$B(k)=P^T(k+1:0)A(k)P(k:0)$. Thus, $B(k)$ is stochastic as the product
of finitely many stochastic matrices is a stochastic matrix.

The proof of relation
$B(k:s)=P^T(k:0)A(k:s)P(s:0)$ proceeds by induction on $k$ for $k> s$ and
an arbitrary but fixed $s\ge0$.
For $k=s+1$, by the definition of $B(s)$ (see Definition~\ref{def:rotation}),
we have $B(s)=P^T(s+1:0)A(s)P(s:0)$, while $B(s+1:s)=B(s)$ and
$A(s+1:s)=A(s)$.
Hence, $B(s+1,s)=P^T(s+1:0)A(s+1:s)P(s:0)$ which shows that
$B(k,s)=P^T(k:0)A(k:s)P(s:0)$ for $k=s+1$, thus implying that
the claim is true for $k=s+1$.

Now, suppose that the claim is true for some $k>s$, i.e.,
$B(k,s)=P^T(k:0)A(k:s)P(s:0)$ for some $k>s$.
Then, for $k+1$ we have
\begin{align}\label{eqn:induction}
B(k+1:s)=B(k)B(k:s)=B(k)\left(P^T(k:0)A(k:s)P(s:0)\right),
\end{align}
where the last equality follows by the induction hypothesis.
By the definition of $B(k)$, we have $B(k)=P^T(k+1:0)A(k)P(k:0)$, and
by replacing $B(k)$ by $P^T(k+1:0)A(k)P(k:0)$ in Eq.~\eqref{eqn:induction},
we obtain
\begin{align*}
B(k+1:s)
&=\left(P^T(k+1:0)A(k)P(k:0)\right)\left( P^T(k:0)A(k:s)P(s:0)\right)\cr
&=P^T(k+1:0)A(k)\left(P(k:0)P^T(k:0)\right)A(k:s)P(s:0)\cr
&=P^T(k+1:0)A(k)A(k:s)P(s:0),
\end{align*}
where the last equality follow from $P^TP=I$ which is valid for any
permutation matrix $P$, and the fact that the product of two
permutation matrices is a permutation matrix.
Since $A(k)A(k:s)=A(k+1:s)$, it follows that
\[B(k+1:s)=P^T(k+1:0)A(k+1:s)P(s:0),\]
thus showing that the claim is true for $k+1$.

\noindent (b)\quad
Let the chain $\Ac$ be ergodic and fix an arbitrary starting time $t_0\geq 0$.
Then, for any $\epsilon>0$, there exists a sufficiently large time
$N_{\epsilon}\geq t_0$, such that the rows of $A(k:t_0)$ are within
$\e$-vicinity of each other; specifically
$\|A_{i}(k:t_0)-A_j(k:t_0)\|\leq \epsilon$ for any $k\geq N_{\epsilon}$ and
all $i,j\in[m]$.
We now look at the matrix $B(k:t_0)$ and its rows. By part (a), we have
for all $k> t_0$,
\[B(k:t_0)=P^T(k:0)A(k:t_0)P(t_0:0).\]
Furthermore, the $i$th row of $B(k:t_0)$ can be represented as
$e_i^TB(k:t_0)$. Therefore, the norm of the difference between
the $i$th and $j$th row of $B(k:t_0)$ is given by
\[\|B_i(k:t_0)-B_j(k:t_0)\|=\|(e_i-e_j)^TB(k:t_0)\|
=\|(e_i-e_j)^TP^T(k:0)A(k:t_0)P(t_0:0)\|.\]
Letting  $e_{i(k)}=P(k:0)e_i$ for any $i\in[m]$, we further have
\begin{align}\label{eqn:ijdisteq}
\|B_i(k:t_0)-B_j(k:t_0)\|
&=\|(e_{i(k)}-e_{j(k)})^TA(k:t_0)P(t_0:0)\|\cr
&=\|(A_{i(k)}(k:t_0)-A_{j(k)}(k:t_0))P(t_0:0)\|\cr
&=\|A_{i(k)}(k:t_0)-A_{j(k)}(k:t_0)\|,
\end{align}
where the last inequality holds since $P(t_0:0)$ is a permutation matrix
and $\|Px\|=\|x\|$ for any permutation $P$ and any $x\in\Rm$.
Choosing $k\ge N_\epsilon$ and using $\|A_{i}(k:t_0)-A_j(k:t_0)\|\leq \epsilon$
for any $k\geq N_{\epsilon}$ and all $i,j\in[m]$,
we obtain
\[\|B_i(k:t_0)-B_j(k:t_0)\|\le\epsilon\qquad
\hbox{for any $k\geq N_{\epsilon}$ and all $i,j\in[m]$.}\]
Therefore, it follows that the ergodicity of $\Ac$ implies the ergodicity of
$\Bc$.

For the reverse implication we note that $A(k:t_0)=P(k:0)B(k:t_0)P^T(t_0:0)$,
which follows by part~(a) and the fact $PP^T=PP^T=I$ for any permutation $P$.
The rest of the proof follows a line of analysis similar to the preceding case,
where we exchange the roles of $B(k:t_0)$ and $A(k:t_0)$.

\noindent (c)\quad Let $\Sc$ be a regular sequence. Let $T(k)$ be the image of
$S(k)$ under the permutation $P^T(k:0)$, i.e.\ $T(k)=P^T(k:0)(S(k))$. Note that
$|T(k)|=|S(k)|$ for all $k\geq 0$ and since $\Sc$ is a regular sequence,
it follows that $\{T(k)\}$ is also a regular sequence.
Now, by the definition of rotational transformation we have
$A(k)=P(k+1:0)B(k)P^T(k:0)$, and hence:
\begin{align}\nonumber
\sum_{i\in S(k+1),j\in \bar{S}(k)}e_i^TA(k)e_j
&=\sum_{i\in S(k+1),j\in \bar{S}(k)}e_i^T[P(k+1:0)B(k)P^T(k:0)]e_j\cr
&=\sum_{i\in T(k+1),j\in \bar{T}(k)}e_i^TB(k)e_j.
\end{align}
Similarly, we have $\sum_{i\in \bar{S}(k+1),j\in S(k)}e_i^TA(k)e_j
=\sum_{i\in \bar{T}(k+1),j\in {T}(k)}e_i^TB(k)e_j$. Next, note that
\begin{align*}
A_{S(k+1)S(k)}(k)&=\sum_{i\in S(k+1),j\in \bar{S}(k)}e_i^TA(k)e_j
+\sum_{i\in \bar{S}(k+1),j\in S(k)}e_i^T A(k)e_j,\cr
B_{T(k+1)T(k)}(k)&=\sum_{i\in T(k+1),j\in \bar{T}(k)}e_i^TB(k)e_j
+\sum_{i\in \bar{T}(k+1),j\in T(k)}e_i^T B(k)e_j.
\end{align*}
Hence, we have $A_{S(k+1)S(k)}(k)=B_{T(k+1)T(k)}(k)$, implying
\[F(\Ac;\Sc)=F(\Bc;\{T(k)\}).\]

For the converse, for any regular sequence $\{T(k)\}$, if we let
$S(k)=P(k:0)(T(k))$ and use the same line of argument as in the preceding case,
we can conclude that $\Sc$ is a regular sequence and
$A_{S(k+1)S(k)}(k)=B_{T(k+1)T(k)}(k)$. Therefore,
$F(\Ac;\Sc)=F(\Bc;\{T(k)\})$ implying that $\Ac$ has the
absolute infinite flow property if and only if $\Bc$ has
the absolute infinite flow property.

\noindent (d)\quad
If $\Pc$ is such that $S(k)$ is the image of $S$ under $\Pc$ at any time
$k\geq 0$, by part (c), it follows that
$A_{S(k+1)S(k)}(k)=B_{T(k+1)T(k)}(k)$, where $T(k)=P^T(k:0)(S(k))$.
Since $S(k)$ is the image of $S$ under $\Pc$ at time $k$, it follows that
$P^T(k:0)(S(k))=S$ in view of $P^T(k:0)P(k:0)=I$.
Thus, $T(k+1)=S$ and by part (c) we obtain $A_{S(k+1)S(k)}(k)=B_S(k)$.
\end{proof}

As listed in Lemma~\ref{lemma:rotationproperties},
the rotational transformation has some interesting properties:
it preserves ergodicity and it preserves the absolute infinite flow property,
by interchanging the flows. We will use these properties intensively in the
development in this section and the rest of the paper.
\subsection{Necessity of Absolute Infinite Flow Property}
\label{sec:necessary}
We now establish the necessity of
the absolute infinite flow property for ergodicity of stochastic chains.
The proof of this result relies on Lemma~\ref{lemma:rotationproperties}.

\begin{theorem}\label{thrm:necessity}
The absolute infinite flow property is necessary for ergodicity of
any stochastic chain.
\end{theorem}
\begin{proof} Let $\Ac$ be an ergodic stochastic chain.
Let $\Sc$ be any regular sequence. Then, there is a permutation sequence
$\Pc$ such that $\Sc$ is the trajectory of the set $S(0)\subset[m]$
under $\Pc$, i.e., $S(k)=P(k:0)(S(0))$ for all $k$, where $P(0:0)=I$.
Let $\Bc$ be the rotational transformation of $\Ac$ by
the permutation sequence $\Pc$. Then, by
Lemma~\ref{lemma:rotationproperties}\eqref{rot:trans}, the chain
$\Bc$ is stochastic. Moreover, by
Lemma~\ref{lemma:rotationproperties}\eqref{rot:equi}, the chain $\Bc$ is
ergodic. Now, by the necessity of the infinite flow property
(Theorem~\ref{thrm:necessaryergodicity}), the chain $\Bc$ should satisfy
\begin{align}\label{eqn:infflowB}
\sum_{k=0}^{\infty}B_{S}(k)=\infty\qquad\hbox{for all nonempty $S\subset[m]$}.
\end{align}
Therefore, 
by Lemma~\ref{lemma:rotationproperties}\eqref{rot:flow},
Eq.~\eqref{eqn:infflowB} with $S=S(0)$ implies
\begin{align*}
\sum_{k=0}^{\infty}A_{S(k+1),{S}(k)}(k)=\infty,
\end{align*}
thus showing that $\Ac$ has the absolute infinite flow property.
\end{proof}

The converse statement of Theorem~\ref{thrm:necessity} is not true generally,
namely the absolute infinite flow need not be sufficient for
the ergodicity of a chain.
We reinforce this statement later in Section~\ref{sec:decomposable}
(Corollary~\ref{cor:counterexample}). Thus, even though
the absolute infinite flow property imposes a lot of structure for
a chain $\Ac$, by requiring that the flow of $\Ac$ over any
regular sequence $\Sc$ be infinite, this is still not enough to guarantee
ergodicity of the chain.
However, as we will soon see, it turns out that this property is sufficient
for ergodicity of the doubly stochastic chains.

\section{Decomposable Stochastic Chains}\label{sec:decomposable}
In this section, we consider a class of stochastic chains, termed decomposable,
for which verifying the absolute infinite flow property can be reduced
to showing that the flows over some specific regular sequences are infinite.
We explore some properties of this class which will be also used in later
sections.

Here is the definition of a decomposable chain.
\begin{definition}\label{def:decomposable}
A chain $\Ac$ is decomposable if $\Ac$ can be represented as a
nontrivial convex combination of a permutation chain $\Pc$ and a stochastic
chain $\Atc$, i.e., there exists a $\gamma>0$ such that
\begin{align}\label{eqn:permutationdecompos}
A(k)=\gamma P(k)+(1-\gamma)\At(k)\qquad \mbox{for all $k\geq 0$}.
\end{align}
We refer to $\Pc$ as a {\it permutation component} of $\Ac$ and to $\gamma$ as
a {\it mixing coefficient} for $\Ac$.
\end{definition}

An example of a decomposable chain is a chain $\Ac$ with uniformly bounded
diagonal entries, i.e., with $A_{ii}(k)\geq \gamma$ for all $k\geq 0$
and some $\gamma>0$. Instances of such chains have been studied
in~\cite{Tsitsiklis84,Jadbabaie03,CaoMora,CarliFagnani06,Krause:02,Boyd06,
alex_cdc08}. In this case, $A(k)=\gamma I +(1-\gamma)\At(k)$ where
$\tilde A(k)=\frac{1}{1-\gamma}(A(k)-\gamma I)$. Note that
$A(k)-\gamma I\geq 0$, where the inequality is to be understood component-wise,
and $(A(k)-\gamma I)e=(1-\gamma)e$, which follows from $A(k)$ being stochastic.
 Therefore, $\At(k)$ is a stochastic matrix for any $k\geq 0$ and
the trivial permutation $\{I\}$ is a permutation component of $\Ac$.
Later, we will show that \textit{any} doubly stochastic chain is decomposable.

An example of a non-decomposable chain is the static chain $\Ac$ defined by
 \[A(k)=A(0)=\left[\begin{array}{cc}
   1&0\\
   1&0
 \end{array}\right]\qquad \mbox{for all $k\geq 0$}.\]
This chain is not decomposable since otherwise, we should have
$A(0)=\gamma P(0)+(1-\gamma)\At(0)$ for some permutation matrix $P(0)$.
Since the second column of $A(0)$ is the zero vector, the second column of
$P(0)$ should also be zero which contradicts
the fact that $P(0)$ is a permutation matrix.

We have some side remarks about decomposable chains. The first remark is
an observation that a permutation component of a decomposable chain $\Ac$ need
not to be unique. An extreme example is the chain $\Ac$ with
$A(k)=\frac{1}{m}ee^T$ for all $k\geq 0$. Since
$\frac{1}{m}ee^T=\frac{1}{m!}\sum_{\xi=1}^{m!}\Qi$,
any sequence of permutation matrices is a permutation component of $\Ac$.
Another remark is about a mixing coefficient $\gamma$ of a chain $\Ac$.
Note that if $\gamma>0$ is a mixing
coefficient for a chain $\Ac$, then any $\xi\in(0,\gamma]$
is also a mixing coefficient for $\Ac$, as it can be seen from
the decomposition in Eq.~\eqref{eqn:permutationdecompos}.

An interesting property of any decomposable chain is that if
such a chain is rotationally transformed with respect to its
permutation component,
the resulting chain has trivial permutation component $\{I\}$. This property
is established in the following lemma.
\begin{lemma}\label{lemma:rotationdecomposable}
Let $\Ac$ be a decomposable chain with a permutation component $\Pc$ and
a mixing coefficient $\gamma$. Let $\Bc$ be the rotational transformation of
$\Ac$ with respect to $\Pc$.
Then, the chain $\Bc$ is decomposable with a trivial permutation component
$\{I\}$ and a mixing coefficient $\gamma$.
\end{lemma}
\begin{proof}
Note that by the definition of a decomposable chain
(Definition~\ref{def:decomposable}), we have
\[A(k)=\gamma P(k)+(1-\gamma)\At(k)\qquad\mbox{for all $k\geq 0$},\]
where $P(k)$ is a permutation matrix and $\At(k)$ is a stochastic matrix.
Therefore,
\[A(k)P(k:0)=\gamma P(k)P(k:0)+(1-\gamma)\At(k) P(k:0).\]
By noticing that $P(k)P(k:0)=P(k+1:0)$ and by using left-multiplication with
$P^T(k+1:0)$, we obtain
\[P^T(k+1:0)A(k)P(k:0)
=\gamma P^T(k+1:0)P(k+1:0)+(1-\gamma)P^T(k+1:0)\At(k) P(k:0).\]
By the definition of the rotational transformation
(Definition~\ref{def:rotation}), we have
$B(k)=P^T(k+1:0)A(k)P(k:0)$. Using this and the fact $P^TP=I$
for any permutation matrix $P$, we further have
\[B(k)=\gamma I +(1-\gamma)P^T(k+1:0)\At(k) P(k:0).\]
Define $\tilde{B}(k)=P^T(k+1:0)\At(k)P(k:0)$ and note that
each $\tilde B(k)$  is a stochastic matrix.  Hence,
\begin{align*}
B(k)=\gamma I+(1-\gamma)\tilde{B}(k),
\end{align*}
thus showing that the chain $\{B(k)\}$ is decomposable with
the trivial permutation component and a mixing coefficient $\gamma$.
\end{proof}

In the next lemma, we prove that the absolute infinite flow property and
the infinite flow property are one and the same for decomposable chains
with a trivial permutation component.
\begin{lemma}\label{lemma:equiinfiniteflows}
For a decomposable chain with a trivial permutation component,
the infinite flow property and the absolute infinite flow property
are equivalent.
\end{lemma}
\begin{proof}
By definition, the absolute infinite flow property implies the infinite flow
property for any stochastic chain. For the reverse implication, let $\Ac$ be
decomposable with a permutation component $\{I\}$. Also, assume that $\Ac$ has
the infinite flow property.
To show that $\Ac$ has the absolute infinite flow property,
let $\Sc$ be any regular sequence.
If $S(k)$ is constant after some time $t_0$, i.e., $S(k)=S(t_0)$ for
$k\geq t_0$ and some $t_0\ge0$, then
\begin{align*}
\sum_{k=t_0}^{\infty}A_{S(k+1),{S}(k)}(k)
=\sum_{k=t_0}^{\infty}A_{S(t_0)}(k) 
=\infty,
\end{align*}
where the last equality holds since $\Ac$ has infinite flow property
and $\sum_{k=0}^{t_0}A_{S(t_0)}(k)$ is finite.
Therefore, if $S(k)=S(t_0)$ for $k\geq t_0$, then we must have
$\sum_{k=0}^{\infty}A_{S(k+1),{S}(k)}(k)=\infty$.

If there is no $t_0\ge0$ with $S(k)=S(t_0)$ for $k\geq t_0$, then
we must have $S(k_r+1)\not=S(k_r)$ for an increasing time sequence $\{k_r\}$.
Now, for an $i\in S(k_r)\setminus S(k_r+1)\not=\emptyset$, we have
$A_{S(k_r+1),{S}(k_r)}(k_r)\geq A_{ii}(k_r)$ since $i\in \bar{S}(k_r+1)$.
Furthermore, $A_{ii}(k)\geq \gamma$ for all $k$
since $\Ac$ has the trivial permutation sequence $\{I\}$ as a permutation
component with a mixing coefficient $\gamma$. Therefore,
\begin{align*}
\sum_{k=0}^{\infty}A_{S(k+1),{S}(k)}(k)
&\geq \sum_{r=0}^{\infty}A_{S(k_r+1),{S}(k_r)}(k_r)
\geq \gamma\sum_{r=0}^{\infty}1=\infty.
\end{align*}
All in all, $F(\Ac,\Sc)=\infty$ for any regular sequence $\Sc$ and, hence,
the chain $\Ac$ has the absolute infinite flow property.
\end{proof}

Lemma~\ref{lemma:equiinfiniteflows} shows that the absolute infinite flow
property may be easier to verify for the chains with a trivial permutation
component, by just checking the infinite flow property. This result,
together with Lemma~\ref{lemma:rotationdecomposable} and the properties of
rotational transformation established in Lemma~\ref{lemma:rotationproperties},
provide a basis to show that a similar reduction of the absolute infinite flow
property is possible for any decomposable chain.
\begin{theorem}\label{thm:reductionabsoluteflow}
Let $\Ac$ be a decomposable chain with a permutation component $\Pc$. Then,
the chain $\Ac$ has the absolute infinite flow property if and only if
$F(\Ac;\Sc)=\infty$ for any trajectory $\Sc$ under $\Pc$, i.e., for all
$S(0)\subset[m]$ and its trajectory $\Sc$ under $\Pc$.
\end{theorem}
\begin{proof}
Since, by the definition, the absolute infinite flow property implies
$F(\Ac;\Sc)=\infty$ for any regular sequence $\Sc$, it suffice to show that
$F(\Ac;\Sc)=\infty$ for any trajectory $\Sc$ under $\Pc$. To show this,
let $\Bc$ be the rotational
transformation of $\Ac$ with respect to $\Pc$. Since $\Ac$ is decomposable, by
Lemma~\ref{lemma:rotationdecomposable}, it follows that $\Bc$ has the trivial
permutation component $\{I\}$. Therefore, by
Lemma~\ref{lemma:equiinfiniteflows} $\Bc$ has
the absolute infinite flow property
if and only if it has the infinite flow property, i.e.,
\begin{align}\label{eqn:flowBk}
\sum_{k=0}^\infty B_{S}(k)=\infty\qquad\hbox{for all nonempty $S\subset[m]$}.
\end{align}
By Lemma~\ref{lemma:rotationproperties}\ref{rot:flow}, we have
$B_S(k)=A_{S(k+1),S(k)}(k)$, where $S(k)$ is the image of
$S(0)=S$ under the permutation $\Pc$ at time $k$. Therefore,
Eq.~\eqref{eqn:flowBk} holds if and only if
\[\sum_{k=0}^\infty A_{S(k+1)S(k)}(k)=\infty,\]
which in view of $F(\Ac;\Sc)=\sum_{k=0}^\infty A_{S(k+1)S(k)}(k)$ shows
that $F(\Ac;\Sc)=\infty$.
\end{proof}


Another direct consequence of Lemma~\ref{lemma:equiinfiniteflows} is that
the absolute infinite flow property is not generally sufficient for ergodicity.

\begin{corollary}\label{cor:counterexample}
The absolute infinite flow property is not a sufficient condition for
ergodicity.
\end{corollary}
\begin{proof}
Consider the following static chain:
\[A(k)=\left[\begin{array}{ccc}
1&0&0\\ \frac{1}{3}&\frac{1}{3}&\frac{1}{3}\\ 0&0&1
\end{array}\right]\qquad\mbox{for $k\geq 0$}.\]
It can be seen that $\Ac$ has the infinite flow property. Furthermore, it can
be seen that $\Ac$ is decomposable and has  the trivial permutation sequence
$\{I\}$ as a permutation component. Thus, by
Lemma~\ref{lemma:equiinfiniteflows}, the chain $\Ac$ has
the absolute infinite flow property. However, $\Ac$ is not ergodic.
This can be seen by noticing that the vector $v=(1,\frac{1}{2},0)^T$ is a
fixed point of the dynamics $x(k+1)=A(k)x(k)$ with $x(0)=v$, i.e.,
$v=A(k)v$ for any $k\geq 0$. Hence, $\Ac$ is not ergodic.
\end{proof}

Although the absolute infinite flow property is a stronger necessary condition
for ergodicity than the infinite flow property,
Corollary~\ref{cor:counterexample} demonstrates that
the absolute infinite flow property is not yet strong enough to be equivalent
to ergodicity. 

\section{Doubly Stochastic Chains}\label{sec:doublystoch}
In this section, we focus on the class of the doubly stochastic chains.
We first show that this class is a subclass of the decomposable chains.
Using this result and the results developed in the preceding sections, we
establish that the absolute infinite flow property is equivalent to ergodicity
for doubly stochastic chains.

We start our development by proving that a doubly stochastic chain is
decomposable. The key ingredient in this development is the Birkhoff-von
Neumann theorem (\cite{HornJohnson} page 527) stating that a doubly stochastic
matrix is a convex combination of permutation matrices.
More precisely, by the Birkhoff-von Neumann theorem,
a matrix $A$ is doubly stochastic if and only if
$A$ is a convex combination of permutation matrices, i.e.,
\[A=\sum_{\xi=1}^{m!}q_{\xi}\Qi,\]
where $\sum_{\xi=1}^{m!}q_{\xi}=1$ and $q_{\xi}\geq 0$ for $\xi\in[m!]$.
Here, we assume that all the permutation matrices are indexed in some manner
(see notation in Section~\ref{sec:introduction}).

Now, consider a sequence $\Ac$ of doubly stochastic matrices.
By applying the Birkhoff-von Neumann theorem to each $A(k)$, we have
\begin{align}\label{eqn:original2}
A(k)=\sum_{\xi=1}^{m!}q_{\xi}(k)\Qi,
\end{align}
where $\sum_{\xi=1}^{m!}q_{\xi}(k)=1$ and $q_{\xi}(k)\geq 0$ for all
$\xi\in[m!]$ and $k\geq 0$.
Since $\sum_{\xi=1}^{m!}q_{\xi}(k)=1$ and $q_{\xi}(k)\geq 0$,
there exists a scalar $\gamma\ge \frac{1}{m!}$ such that
for every $k\geq 0$, we can find $\xi(k)\in[m!]$ satisfying
$q_{\xi(k)}(k)\geq \gamma$.
Therefore, for any time $k\geq 0$, there is a permutation matrix
$P(k)=P^{(\xi(k))}$ such that
\begin{align}\label{eqn:originaldoubly}
A(k)=\gamma P(k)+\sum_{\xi=1}^{m!}\al_{\xi}(k)\Qi=\gamma P(k)+(1-\gamma)\At(k),
\end{align}
where $\gamma>0$ is a \textit{time-independent} scalar and
$\At(k)=\frac{1}{1-\gamma}\sum_{\xi=1}^{m!}\al_{\xi}(k)\Qi$.

The decomposition of $A(k)$ in Eq.~\eqref{eqn:originaldoubly}
fits the description in the definition of decomposable chains.
Therefore, we have established the following result.
\begin{lemma}\label{lemma:doublystochsticdecom}
Any doubly stochastic chain is decomposable.
\end{lemma}

In the light of Lemma~\ref{lemma:doublystochsticdecom},
all the results developed in Section~\ref{sec:decomposable} are applicable to
doubly stochastic chains. Another result that we use is the special instance of
Theorem~6 in~\cite{ErgodicityPaper} as applied to doubly stochastic chains.
Any doubly stochastic chain that has the trivial permutation component $\Ic$
(i.e., Eq.~\eqref{eqn:originaldoubly} holds with $P(k)=I$) fits the framework
of Theorem~6 in~\cite{ErgodicityPaper}.
We restate this theorem for convenience.
\begin{theorem}\label{thrm:ergodicityfeedback}
Let $\Ac$ be a doubly stochastic chain with a permutation component $\{I\}$.
Then, the chain $\Ac$ is ergodic if and only if it has
the infinite flow property.
\end{theorem}

Now, we are ready to deliver our main result of this section,
showing that the ergodicity and the absolute infinite flow are equivalent
for doubly stochastic chains. We accomplish this by combining
Theorem~\ref{thm:reductionabsoluteflow} and
Theorem~\ref{thrm:ergodicityfeedback}.

\begin{theorem}\label{thrm:infflowdoubly}
A doubly stochastic chain $\Ac$ is ergodic if and only if it has the absolute
infinite flow property.
\end{theorem}
\begin{proof}
Let $\Pc$ be a permutation component for $\Ac$ and let $\Bc$ be the rotational
transformation of $\Ac$ with respect to its permutation component.
By Lemma~\ref{lemma:rotationdecomposable}, $\Bc$ has the trivial permutation
component $\Ic$. Moreover, since $B(k)=P^T(k+1:0)A(k)P(k:0)$, where
$P^T(k+1:0)$, $A(k)$ and $P(k:0)$ are doubly stochastic matrices, it follows
that $\Bc$ is a doubly stochastic chain. Therefore, by
Theorem~\ref{thrm:ergodicityfeedback}, it follows that $\Bc$ is ergodic if
and only if it has the infinite flow property.
Then, by Lemma~\ref{lemma:rotationproperties}\eqref{rot:flow},
the chain $\Bc$ has the infinite flow property if and only if $\Ac$ has
the absolute infinite flow property.
\end{proof}

Theorem~\ref{thrm:infflowdoubly} provides an \textit{alternative
characterization} of ergodicity for doubly stochastic chains, under only
requirement to have the absolute infinite flow property.
We note that Theorem~\ref{thrm:infflowdoubly}
does not impose any other specific conditions on matrices $A(k)$ such
as uniformly bounded diagonal entries or uniformly bounded positive entries,
which have been typically assumed in the existing literature
(see for example \cite{Tsitsiklis84,Jadbabaie03,CaoMora,CarliFagnani06,
Krause:02,Boyd06,alex_cdc08}).

We observe that the absolute infinite flow typically requires verifying that
the infinite flow exists along every regular sequence of index sets.
However, to use Theorem~\ref{thrm:infflowdoubly},
we do not have to check the infinite flow for every regular sequence.
This reduction in checking the absolute infinite flow property is due to
Theorem~\ref{thm:reductionabsoluteflow}, which shows that in order to assert
the absolute infinite flow property for doubly stochastic chains,
it suffices that the flow over some specific regular sets is infinite.
We summarize this observation in the following corollary.

\begin{corollary}\label{cor:simpleabsoluteflow}
Let $\Ac$ be a doubly stochastic chain with a permutation component $\Pc$.
Then, the chain is ergodic if and only if $F(\Ac;\Sc)=\infty$ for all
trajectories $\Sc$ of subsets $S(0)\subset [m]$ under $\Pc$.
\end{corollary}

\subsection{Rate of Convergence}\label{sec:rateofconvergence}
Here, we explore a rate of convergence result for an ergodic doubly
stochastic chain $\Ac$. In the development, we adopt a dynamical system
point of view for the chain and we consider a Lyapunov function
associated with the dynamics. The final ingredient in the development
is the establishment of another important property of rotational transformation
related to the invariance of the Lyapunov function.

Let $\Ac$ be a doubly stochastic chain and consider the dynamic system driven
by this chain. Specifically, define the following dynamics, starting with any
initial point $x(0)\in\Rm$,
\begin{align}\label{eqn:dynsys}
x(k+1)=A(k)x(k)\qquad\mbox{for $k\geq 0$}.
\end{align}
With this dynamics, we associate a Lyapunov function $V(\cdot)$ defined
as follows:
\begin{align}\label{eqn:lyapunov}
V(x)=\sum_{i=1}^m(x_i-\bar{x})^2\qquad\hbox{for $x\in\Rm$},
\end{align}
where $\bar{x}=\frac{1}{m}e^Tx$.
This function has often been used in studying ergodicity of doubly stochastic
chains (see for example~\cite{Boyd06, Zampieri,alex_cdc08,ErgodicityPaper}).

We now consider the behavior of the Lyapunov function under rotational
transformation of the chain $\Ac$, as given in Definition~\ref{def:rotation}.
It emerged that the Lyapunov function $V$ is invariant under the rotational
transformation, as shown in forthcoming Lemma~\ref{lemma:rotlyap}.
We emphasize that {\it the invariance of the Lyapunov function $V$
holds for arbitrary stochastic chain}; the doubly stochastic assumption is not
needed for the underlying chain.

\begin{lemma}\label{lemma:rotlyap}
Let $\Ac$ be a stochastic chain and $\Pc$ be an arbitrary
permutation chain. Let $\Bc$ be the rotational transformation of $\Ac$ by
$\Pc$. Let $\xc$ and $\{y(k)\}$ be the dynamics obtained by $\Ac$ and $\Bc$,
respectively, with the same initial point $y(0)=x(0)$ where $x(0)\in\Rm$ is
arbitrary. Then, for the function $V(\cdot)$ defined
in Eq.~\eqref{eqn:lyapunov} we have
\[V(x(k))=V(y(k))\qquad\hbox{for all $k\geq 0$}.\]
\end{lemma}
\begin{proof}
Since $\{y(k)\}$ is the dynamic obtained by $\Bc$, there holds
for any $k\ge0$,
\[y(k)=B(k-1)y(k-1)=\ldots= B(k-1)\cdots B(1)B(0) y(0)=B(k:0)y(0).\]
By Lemma~\ref{lemma:rotationproperties}\eqref{rot:trans}, we have
$B(k:0)=P^T(k:0)A(k:0)P(0:0)$ with $P(0:0)=I$, implying
\begin{align}\label{eqn:relxy}
y(k)=P^T(k:0)A(k:0) y(0)=P^T(k:0)A(k:0)x(0)=P^T(k:0)x(k),
\end{align}
where the second equality follows from $y(0)=x(0)$ and the last equality
follows from the fact that $\xc$ is the dynamic obtained by $\Ac$.
Now, notice that the function $V(\cdot)$
of~Eq.~\eqref{eqn:lyapunov} is invariant under any permutation, that is
$V(Px)=V(x)$ for any permutation matrix $P$.
In view of  Eq.~\eqref{eqn:relxy}, the vector
$y(k)$ is just a permutation of $x(k)$. Hence,
$V(y(k))=V(x(k))$ for all $k\geq 0$.
\end{proof}

Consider an ergodic doubly stochastic chain $\Ac$ with a trivial permutation
component $\{I\}$. Let $t_0=0$ and for any $\delta\in(0,1)$ recursively define
$t_q$, as follows:
\begin{align}\label{eqn:tq}
t_{q+1}=\arg\min_{t\geq t_q+1}\ \min_{S\subset[m]}\
\sum_{t=t_q}^{t-1}A_{S}(k)\geq \delta,
\end{align}
where the second minimum in the above expression is taken over all nonempty
subsets $S\subset [m]$. Basically, $t_q$ is the first time $t> t_{q-1}$
when the accumulated flow from $t=t_{q-1} +1$ to $t= t_q$ exceeds $\delta$
over every nonempty $S\subset[m]$. We refer to the sequence $\{t_q\}$ as
\textit{accumulation times} for the chain $\Ac$. We observe that, when
the chain $\Ac$ has infinite flow property, then
$t_q$ exists for all $q\geq 0$, and any $\delta>0$.

Now, for the sequence of time
instances $\{t_q\}$, we have the following rate of convergence result.

\begin{lemma}\label{lemma:rateofconvtrivial}
Let $\Ac$ be an ergodic doubly stochastic chain with a trivial permutation
component $\{I\}$ and a mixing coefficient $\gamma>0$. Also, let
$\xc$ be the dynamics driven by $\Ac$ starting at an arbitrary point $x(0)$.
Then, for all $q\geq 1$, we have
\[V(x(t_q))\leq
\left(1-\frac{\gamma\delta(1-\delta)^2}{m(m-1)^2}\right)V(x(t_{q-1})),\]
where $t_q$ is defined in \eqref{eqn:tq}.
\end{lemma}
The proof of Lemma~\ref{lemma:rateofconvtrivial} is based on the proof of
Theorem~10 in~\cite{NedichOlshevskyOzdaglar:09} and Theorem~6
in~\cite{ErgodicityPaper}, and it is provided in Appendix.

Using the invariance of the Lyapunov function under rotational transformation
and the properties of this transformation, we can establish a result
analogous to Lemma~\ref{lemma:rateofconvtrivial}
for an arbitrary ergodic chain of doubly stochastic matrices $\Ac$. In other
words, we can extend Lemma~\ref{lemma:rateofconvtrivial} to the case when the
chain $\Ac$ does not necessarily have the trivial permutation component
$\{I\}$. To do so, we appropriately adjust the definition of
the accumulation times $\{t_q\}$ for this case. In particular, we let
$\delta>0$ be arbitrary but fixed, and let $\Pc$ be a permutation component of
an ergodic chain $\Ac$.
Next, we let $t_0=0$ and for $q\ge1$, we define $t_q$ as follows:
\begin{align}\label{eqn:tqgen}
t_{q+1}=\arg\min_{t\geq t_q+1}\ \min_{S(0)\subset[m]}\
\sum_{t=t^{\delta}_q}^{t-1}A_{S(k+1)S(k)}(k)\geq \delta,
\end{align}
where $\Sc$ is the trajectory of the set $S(0)$ under $\Pc$.

We have the following convergence result.
\begin{theorem}\label{thrm:rateofconv}
Let $\Ac$ be an ergodic doubly stochastic chain with a permutation component
$\Pc$ and a mixing coefficient $\gamma>0$. Also, let $\xc$ be the dynamics
driven by $\Ac$ starting at an arbitrary point $x(0)$. Then, for all $q\geq 1$,
we have
\[V(x(t_q))\leq
\left(1-\frac{\gamma\delta(1-\delta)^2}{m(m-1)^2}\right)V(x(t_{q-1})),\]
where $t_q$ is defined in \eqref{eqn:tqgen}.
\end{theorem}
\begin{proof}
Let $\Bc$ be the rotational transformation of the chain $\Ac$ with respect to
$\Pc$. Also, let $\yc$ be the dynamics driven by chain $\Bc$ with the initial
point $y(0)=x(0)$. By Lemma~\ref{lemma:rotationdecomposable}, $\Bc$ has the
trivial permutation component $\{I\}$. Thus,
by Lemma~\ref{lemma:rateofconvtrivial}, we have for all $q\ge1,$
\[V(y(t_q))\leq \left(1-\frac{\gamma\delta(1-\delta)^2}{m(m-1)^2}\right)
V(y(t_{q-1})).\]
Now, by Lemma~\ref{lemma:rotationproperties}\eqref{rot:flow}, we have
$A_{S(k+1)S(k)}(k)=B_S(k)$. Therefore, the accumulation times for the chain
$\Ac$ are the same as the accumulation times for the chain $\Bc$. Furthermore,
according Lemma~\ref{lemma:rotlyap},  we have
$V(y(k))=V(x(k))$ for all $k\geq 0$ and, hence, for all $q\ge1$,
\[V(x(t_q))\leq \left(1-\frac{\gamma\delta(1-\delta)^2}{m(m-1)^2}\right)
V(x(t_{q-1})).\]
\end{proof}

\subsection{Doubly Stochastic Chains without Absolute Flow Property}
\label{sec:infgraph}
So far we have been concerned with doubly stochastic chains with the absolute
infinite flow property. In this section, we turn our attention to the
case when the absolute flow property is absent.
In particular, we are interested
in characterizing the limiting behavior of backward product of a doubly
stochastic chain that does not have the absolute infinite flow.

Since a doubly stochastic chain is decomposable,
Theorem~\ref{thm:reductionabsoluteflow} is applicable, so by this
theorem when the chain $\Ac$ does not have
the absolute infinite flow property, there holds
$F(\Ac;\Sc)<\infty$ for some $S(0)\subset[m]$ and its trajectory under
a permutation component $\Pc$ of $\Ac$. This permutation component will be
important so we denote it by $\mathcal{P}$.
Furthermore, for this permutation component,
let $\Ginf_{\mathcal{P}}=([m],\Einf_{\mathcal{P}})$ be the undirected
graph with the edge set $\Einf_{\mathcal{P}}$ given by
\[\Einf_{\mathcal{P}}=\left\{\{i,j\} \,\big|\,
\sum_{k=0}^{\infty}A_{i(k+1),j(k)}(k)=\infty\right\},\]
where $\{i(k)\}$ and $\{j(k)\}$ are the trajectories of the sets $S(0)=\{i\}$
and $S(0)=\{j\}$, respectively, under the permutation component $\Pc$;
formally, $e_{i(k)}=P(k:0)e_i$ and $e_{j(k)}=P(k:0)e_j$ for all $k$
with $P(0:0)=I$.
We refer to the graph $\Ginf_{\mathcal{P}}=([m],\Einf_{\mathcal{P}})$ as the
\textit{infinite flow graph} of the chain $\Ac$.
When the permutation component $\Pc$ is trivial, we use $\Ginf$ to denote
$\Ginf_{\mathcal{P}}$.

By Theorem~\ref{thm:reductionabsoluteflow}, we have that $\Ginf_{\mathcal{P}}$
is connected if and only if the chain has the absolute infinite flow property.
Theorem~6 in~\cite{TouriNedich:Approx} shows that for a chain with the trivial
permutation component $\{I\}$, the connectivity of $\Ginf$ is closely related
to the limiting matrices of the product $A(k:t_0)$, as $k\to\infty$
(in the case of the trivial permutation component $\{I\}$,
we have $i(k)=i$ for all $i\in[m]$ and $k\geq 0$).
The following result is just an implication of Theorem~6
in~\cite{TouriNedich:Approx} for the special case of doubly
stochastic chains that have the trivial permutation component.

\begin{theorem}\label{thrm:infgraphtrivial}
Let $\Ac$ be a doubly stochastic chain with the trivial permutation component
$\{I\}$. Then, for any starting time $t_0$, the limit
$A^{\infty}(t_0)=\lim_{k\rightarrow \infty}A(k:t_0)$ exists. Moreover,
the $i$th row and the $j$th row of $A^{\infty}(t_0)$ are identical for any
$t_0\geq 0$, i.e., $\lim_{k\rightarrow\infty}\|A_i(k:t_0)-A_j(k:t_0)\|=0$,
if and only if $i$ and $j$ belong to the same connected component of $\Ginf$.
\end{theorem}

Using Lemma~\ref{lemma:rotationproperties},
Lemma~\ref{lemma:rotationdecomposable}
and Theorem~\ref{thrm:infgraphtrivial}, we prove a similar result for the case
of a general doubly stochastic chain.
\begin{theorem}\label{thrm:infgraphgeneral}
Let $\Ac$ be a doubly stochastic chain with a permutation component $\Pc$.
Then, for any starting time $t_0\geq 0$, the product $A(k:t_0)$ converges up to
a permutation of its rows; i.e.,
there exists a permutation sequence $\{Q(k)\}$ such that
$\lim_{k\rightarrow\infty}Q(k)A(k:t_0)$ exists for any $t_0\ge0$.
Moreover, for the trajectories $\{i(k)\}$ and $\{j(k)\}$ of $S(0)=\{i\}$ and
$S(0)=\{j\}$, respectively, under the permutation component $\Pc$, we have
\[\lim_{k\rightarrow\infty}\|A_{i(k)}(k:t_0)-A_{j(k)}(k:t_0)\|=0\qquad
\mbox{for any starting time $t_0$},\]
if and only if $i$ and $j$ belong to the same connected component of
$\Ginf_{\mathcal{P}}$.
\end{theorem}
\begin{proof}
Let $\Bc$ be the rotational transformation of $\Ac$ by the permutation
component $\Pc$. As proven in Lemma~\ref{lemma:rotationdecomposable},
the chain $\Bc$ has the trivial permutation component. Hence, by
Theorem~\ref{thrm:infgraphtrivial},
the limit $B^{\infty}(t_0)=\lim_{k\rightarrow\infty}B(k:t_0)$ exists
for any $t_0\geq 0$. On the other hand, by
Lemma~\ref{lemma:rotationproperties}\eqref{rot:trans}, we have
\[B(k:t_0)=P^T(k:0)A(k:t_0)P(t_0:0)\qquad\hbox{for all $k>t_0$}.\]
Multiplying by $P^T(t_0:0)$ from the right, and using $PP^T=I$ which is valid
for any permutation matrix $P$, we obtain
\[B(k:t_0)P(t_0:0)^T=P^T(k:0)A(k:t_0)\qquad\hbox{for all $k>t_0$}.\]
Therefore, $\lim_{k\rightarrow\infty}B(k:t_0)P(t_0:0)^T$ always exists for any
starting time $t_0$ since $B(k:t_0)P^T(t_0:0)$ is obtained
by a fixed permutation of the columns of $B(k:t_0)$. Thus, if we let
$Q(k)=P^T(k:0)$, then $\lim_{k\rightarrow\infty}Q(k)A(k:t_0)$ exists for any
$t_0$ which proves the first part of the theorem.

For the second part, by Theorem~\ref{thrm:infgraphtrivial}, we have
$\lim_{k\rightarrow\infty}\|B_i(k:t_0)-B_j(k:t_0)\|=0$ for any $t_0\geq 0$
if and only if $i$ and $j$ belong to the same connected component
of the infinite flow graph of $\Bc$.
By the definition of the rotational
transformation, we have $B(k:t_0)=P^T(k:0)A(k:t_0)P(t_0:0)$.
Then, for the $i$th and $j$th row of $B(k:t_0)$, we have
according to Eq.~\eqref{eqn:ijdisteq}:
\[\|B_i(k:t_0)-B_j(k:t_0)\|
=\|A_{i(k)}(k:t_0)-A_{j(k)}(k:t_0)\|,\]
where $e_{i(k)}=P(k:0)e_i$ and $e_{j(k)}=P(k:0)e_j$ for all $k$.
Therefore, $\lim_{k\rightarrow\infty}\|B_i(k:t_0)-B_j(k:t_0)\|=0$
if and only if
$\lim_{k\rightarrow\infty}\|A_{i(k)}(k:t_0)-A_{j(k)}(k:t_0)\|=0$. Thus,
$\lim_{k\rightarrow\infty}\|A_{i(k)}(k:t_0)-A_{j(k)}(k:t_0)\|=0$ for any
$t_0\geq 0$ if and only if $i$ and $j$ belong to the same connected component
of the infinite flow graph of $\Bc$. The last step is to show that
the infinite flow graph of $\Bc$ and $\Ac$ are the same. This
however, follows from
\[\sum_{k=0}^\infty B_{ij}(k)
=\sum_{k=0}^{\infty}e_i^T[P^T(k+1:0)A(k)P(k:0)]e_j=
\sum_{k=0}^{\infty}A_{i(k+1),j(k)}(k).\]
\end{proof}

By Theorem~\ref{thrm:infgraphgeneral}, for any doubly stochastic chain $\Ac$
and any fixed $t_0$, the sequence consisting of the rows of
$A(k:t_0)$ converges to an \textit{ ordered set of $m$ points}
in the probability simplex of $\Rm$, as $k$ approaches to infinity.
In general, this is not true for an arbitrary stochastic chain. For example,
consider the stochastic chain
\begin{align*}
A(2k)=\left[\begin{array}{ccc}
1&0&0\\ 1&0&0\\ 0&0&1
\end{array}\right],\quad A(2k+1)=\left[\begin{array}{ccc}
1&0&0\\ 0&0&1\\ 0&0&1
\end{array}\right]\qquad\mbox{for all $k\geq 0$}.
\end{align*}
For this chain, we have $A(2k:0)=A(2k)$ and $A(2k+1:0)=A(2k+1)$. Hence,
depending on the parity of $k$, the set consisting of the rows of $A(k:0)$
alters between $\{(1,0,0),(1,0,0),(0,0,1)\}$ and $\{(1,0,0),(0,0,1),(0,0,1)\}$
and, hence, never converges to a unique ordered set of 3 points in $\R^3$.

\section{Application to Switching Systems}\label{sec:switch}
Motivated by the work in~\cite{Gurvits}, here we study
an application of the developed results to the stability analysis
of switching systems that are driven by a subset of stochastic
or doubly stochastic matrices.

Consider an arbitrary collection $\M$ of $m\times m$ stochastic
matrices.
Let us say that a chain $\Ac$ is in $\M$ if $A(k)\in \M$ for all $k\geq 0$.
Our main goal is to characterize conditions on the collection $\M$ such that
\begin{align}\label{eqn:AAS}
\lim_{k\to\infty}A(k:0)=\frac{1}{m}ee^T,
\end{align}
for any chain $\Ac$ in $\M$.
If $\Ac$ is a doubly stochastic chain and we let
$\tilde{A}(k)=A(k)-\frac{1}{m}ee^T$, then we have
$\tilde{A}(k:0)=A(k:0)-\frac{1}{m}ee^T$.
Therefore, relation~\eqref{eqn:AAS} holds for any doubly stochastic chain
$\Ac$ in $\M$ if and only if the (translated) discrete linear inclusion system
$\tilde{\M}=\{A\in \M\mid A-\frac{1}{m}ee^T\}$ is absolutely
asymptotically stable as defined in \cite{Gurvits}.

Also, note that if the relation~\eqref{eqn:AAS} holds for any chain $\Ac$ in
$\M$, then we should have $\lim_{k\to\infty}A(k:t_0)=\frac{1}{m}ee^T$
for all $t_0\geq 0 $ and $\Ac$ in $\M$. This simply follows from the fact that
if $\Ac$ is in $\M$, then any truncated chain $\{A(k)\}_{k\geq t_0}$ of
$\Ac$ should also be in $\M$. Thus, relation~\eqref{eqn:AAS} holds for all
chains $\Ac$ in $\M$ if and only if $\M$ only consists of ergodic chains.

Based on the preceding discussion and the definition of the
absolute asymptotic stability for discrete linear inclusion
systems~\cite{Gurvits}, let us define the absolute asymptotic stability
for a collection of stochastic matrices.
\begin{definition}\label{def:AAS}
We say that a collection $\M$ of stochastic matrices is absolutely
asymptotically stable if every chain $\Ac$ in $\M$ is ergodic.
\end{definition}

To derive necessary and sufficient conditions for the
absolute asymptotic stability, we define the \textit{zero-flow graph}
associated with the collection $\M$. This graph is defined over
a vertex set consisting of all subsets of the set $[m]$, which is denoted
by $\Ps([m])$.
\begin{definition}
Given a collection $\M$, let $G_0(\M)=(\Ps([m]),\E_0(M))$
be a directed graph with the vertex set $\Ps([m])$
and the edge set $\E_0(\M)$ defined by:
  \begin{align}\nonumber
  \E_0(\M)=\{(S,T)\mid S,T\subset [m], |S|=|T|\not=0, \inf_{A\in\M}A_{TS}=0\}.
  \end{align}
 \end{definition}
In other words, an edge $(S,T)$ is in $\E_0(\M)$ if the subsets
$S,T\subset[m]$ are non-trivial and have the same cardinality,
and more importantly, there exists a sequence of matrices
$A^{(k)}\in \M$ such that $A^{(k)}_{ST}\to 0$ as $k$ goes to infinity.

For the forthcoming discussion, let us say that the collection $\M$ has
the absolute infinite flow property if every chain $\Ac$ in $\M$ has
the absolute infinite flow property.
The following result provides a link between
the absolute infinite flow property and the non-existence of cycle in
the graph $G_0(\M)$.
\begin{lemma}\label{lemma:G0absoluteinf}
    A collection $\M$ has the absolute infinite flow property if and only if
its zero-flow graph $G_0(\M)$ is cycle-free.
  \end{lemma}
\begin{proof}
Suppose that $G_0(M)$ has a cycle $S=S_0\to \cdots \to S_r=S$ where
$r\geq 1$ and  $(S_\tau,S_{\tau+1})\in \E_0(\M)$ for all
$\tau\in \{0,\ldots,r-1\}$.
By the definition of the zero-flow graph, for each
$\tau\in\{0,\ldots,r-1\}$, there exists a sequence of matrices
$\{A^{(\tau)}(k)\}$ in $\M$ such that
\[A^{(\tau)}_{S_{\tau+1}S_{\tau}}(k)\leq \frac{1}{2^k}.\]

Now, consider the chain
\[\{\ldots,A^{(1)}(1),A^{(0)}(1),A^{(r-1)}(0),\ldots,A^{(1)}(0),A^{(0)}(0)\},\]
or more precisely, the chain $\Ac$ defined by
\[A(k)
=A^{(k-\lfloor\frac{k}{r}\rfloor)}(\lfloor\frac{k}{r}\rfloor),\]
where $\lfloor a\rfloor$ is the largest integer that is less than
or equal to $a$.
Next, consider the regular sequence $\{T(k)\}$ defined by
\[T(k)=S_{k-\lfloor\frac{k}{r}\rfloor}\qquad\mbox{ for all $k\geq 0$}.\]
Note that, by the definition of the zero-flow graph,
if $S_0\to S_1\to \cdots \to S_r$ is a cycle in $G_0(\M)$
it follows that $|S_0|=\cdots=|S_{r-1}|$ and hence,
$\{T(k)\}$ is a regular sequence.
For the chain $\Ac$ and the regular sequence $\{T(k)\}$, we have:
\begin{align*}
      F(\Ac;\{T(k)\})&=\sum_{k=0}^{\infty}A_{T(k+1)T(k)}(k)\cr
      &=
\sum_{k=0}^{\infty}A^{(k-\lfloor\frac{k}{r}\rfloor)}_{k+1-
\lfloor\frac{k}{r}\rfloor,k-\lfloor\frac{k}{r}\rfloor}
(\lfloor\frac{k}{r}\rfloor)\cr
&=\sum_{k=0}^{\infty}\sum_{\tau=0}^{r-1}\frac{1}{2^k}=2r<\infty.
    \end{align*}
Thus, $\Ac$ does not have the absolute infinite flow property.

Now, assume that $\M$ does not have the absolute infinite flow property.
Then, we have $F(\Ac;\Tc)<\infty$ for some chain $\Ac$ in $\M$ and
a regular sequence $\Tc$. Consider the complete directed graph
$K_{T(0)}$, associated with the set $T(0)$ and given by
$K_{T(0)}=(V_{T(0)}, V_{T(0)}\times V_{T(0)})$ where
vertex set $V_{T(0)}$ is the set of all subsets of $[m]$
with cardinality $|T(0)|$.
Now for the regular set $\Tc$, consider the infinite walk
$T(0)\to T(1)\to T(2)\to\cdots$ in $K_{T(0)}$. Since
$K_{T(0)}$ is a finite graph, it has finitely many cycles and
one of its cycles must be traced infinitely many times by this
infinite length walk. Let $S_0\to S_1 \to\cdots\to S_{r-1}\to S_{r}=S_0$
be such a cycle. Then, there exists an increasing sequence of time instances
$0\leq t_0<t_1<\cdots$ such that $T(t_k+\tau)=S_\tau$
for all $\tau\in\{0,\ldots,r\}$ and all $k\geq 0$. Therefore,
     \begin{align}\nonumber
        \sum_{k=0}^{\infty}\sum_{\tau=0}^{r-1}A_{S_{\tau+1}S_{\tau}}(t_k+\tau)&
        =\sum_{k=0}^{\infty}
       \sum_{\tau=0}^{r-1}A_{T(t_k+\tau+1)T(t_k+\tau)}(t_k+\tau)\cr
        &\leq \sum_{k=0}^{\infty}A_{T(k+1)T(k)}(k)\leq F(\Ac;\Sc)<\infty.
     \end{align}
As a result, we have $\lim_{k\to\infty}A_{S_{\tau+1}S_{\tau}}(t_k+\tau)=0$ for
all $\tau\in \{0,\ldots,r-1\}$, implying that
$(S_{\tau},S_{\tau+1})\in\E_0(\M)$ for all $\tau$ and hence,
$G_0(\M)$ contains a cycle.
  \end{proof}

Based on Lemma~\ref{lemma:G0absoluteinf} and the results established in
the preceding sections, the following result emerges immediately.

\begin{theorem}\label{thrm:aasnecessity}
 A collection $\M$ is absolutely asymptotically stable
 only if its zero-flow graph $G_0(\M)$ is cycle-free.
\end{theorem}
  \begin{proof}
  By Theorem~\ref{thrm:necessity}, $\M$ is absolutely asymptotically stable
only if $M$ has the absolute infinite flow property.
At the same time, by Lemma~\ref{lemma:G0absoluteinf}, $\M$ has the
absolute infinite flow property if and only if $G_0(\M)$ is cycle-free.
  \end{proof}

Now, let us focus on the case when the collection $\M$
contains only doubly stochastic matrices. In this case, the
reverse implication of Theorem~{thrm:aasnecessity}result also holds.
\begin{theorem}\label{thrm:aasnecessityandsufficient}
A collection $\M$ of doubly stochastic matrices is absolutely
asymptotically stable if and only if its zero-flow graph
$G_0(\M)$ is cycle-free.
\end{theorem}
  \begin{proof}
By Theorem~\ref{thrm:infflowdoubly}, $\M$ is absolutely
asymptotically stable if and only if $\M$ has the infinite flow property,
which by Lemma~\ref{lemma:G0absoluteinf} holds if and only if $G_0(\M)$
is cycle-free.
\end{proof}

Suppose that, in addition, we further restrict the matrices in the collection
$\M$ to the matrices resulting in chains $\Ac$ with trivial permutation component $\{I\}$. Then, we can have an alternative way of asserting
the absolute asymptotic stability of $\M$.
\begin{theorem}
Let $\M$ be a collection of doubly stochastic matrices
such that any chain in $\M$ has the trivial permutation component $\{I\}$.
Then, $\M$ is absolutely asymptotically stable if and only if
$G_0(\M)$ is loop-less, which in turn holds if and only if
$\inf_{A\in \M,S\subset[m]}A_{S}>0$.
\end{theorem}
\begin{proof}
Let $\Ac$ be a chain in $\M$ with the trivial permutation component. Then, by Theorem~\ref{thrm:ergodicityfeedback}, such a chain is ergodic if and only if it has the infinite flow property, i.e.\ $\sum_{k=0}^{\infty}A_{SS}(k)=\infty$ for any non-trivial $S\subset[m]$. Using the same lines of argument as in the proof of Theorem~\ref{thrm:aasnecessity}, it can be shown that this happens if and only if $\inf_{A\in \M}A_{SS}>0$ for any non-trivial $S\subset[m]$. But this happens if and only if $G_0(\M)$ is loop-less.
\end{proof}

\section{Conclusion}\label{sec:conclusion}
In this paper, we studied backward product of a chain of stochastic and doubly
stochastic matrices. We introduced the concept of absolute infinite flow
property and studied its relation to ergodicity of the chains.
In our study, a rotational transformation of a chain and the properties
of the transformation played important roles in the development.
We showed that this transformation preserves many properties of the original
chain. Based on the properties of rotational transformation,
we proved that absolute infinite flow property is necessary for ergodicity of
a stochastic chain. Moreover, we showed that this property is also sufficient
for ergodicity of doubly stochastic chains. Then, we developed a rate of
convergence for ergodic doubly stochastic chains. Moreover, we characterized
the limiting behavior of a doubly stochastic chain in the absence of
ergodicity. Finally, using the established results, we derived necessary and
sufficient conditions for the stability of switching systems
driven by stochastic and doubly stochastic matrices.

\appendix
Proof of Lemma~\ref{lemma:rateofconvtrivial}.
\begin{proof}
We prove the assertion for $q=1$. The case of an arbitrary $q\geq 0$ follows
similarly. Let $x(0)\in\Rm$ be an arbitrary starting point and let $\xc$
be the dynamics driven by the doubly stochastic chain $\Ac$. Without loss of
generality, assume that $x_1(0)\leq \ldots \leq x_m(0)$.
By the definition of $t_1$, we have $\sum_{k=0}^{t_1-1}A_{S}(k)\geq \delta$
for any $S\subset[m]$. Thus by Lemma~12 in \cite{ErgodicityPaper}, we have
\begin{align}\label{eqn:rate1}
\sum_{k=0}^{t_1-1}\sum_{i<j}\
&\left(A_{ij}(k)+A_{ji}(k)\right)(x_{i}(k)-x_j(k))^2\geq
\frac{\delta(1-\delta)^2}{x_m(0)-x_1(0)}
\sum_{\ell=1}^{m-1}(x_{\ell+1}(k)-x_{\ell}(k))^3.
\end{align}
On the other hand, by Lemma~11 in \cite{ErgodicityPaper}, we have
\begin{align}\label{eqn:rate2}
V(x(0))&=\sum_{i=1}^m(x_i(0)-\bar{x}(0))^2\cr
&\leq m(x_m(0)-x_1(0))^2\cr
&\leq m(m-1)\sum_{\ell=1}^{m-1}(x_{\ell+1}(0)-x_{\ell}(0))^2\cr
&\leq m(m-1)^2\frac{1}{x_m(0)-x_1(0)}
\sum_{\ell=1}^{m-1}(x_{\ell+1}(0)-x_{\ell}(0))^3.
\end{align}

Therefore, using Eq.~\eqref{eqn:rate1} and Eq.~\eqref{eqn:rate2}, we obtain
\begin{align}\label{eqn:rate3}
\sum_{k=0}^{t_1-1}\sum_{i<j}\
&\left(A_{ij}(k)+A_{ji}(k)\right)(x_{i}(k)-x_j(k))^2
\geq \frac{\delta(1-\delta)^2}{m(m-1)^2}V(x(0)).
\end{align}
Now, by Lemma~4 in \cite{NedichOlshevskyOzdaglar:09}, we have that
\[V(x(k))=V(x(k-1))-\sum_{i<j}H_{ij}(k-1)(x_i(k-1)-x_j(k-1))^2,\]
where $H(k)=A^T(k)A(k)$. Thus, it follows
\begin{align}\label{eqn:rate4}
V(x(t_1))&= V(x(0))-\sum_{k=0}^{t_1-1}\sum_{i<j}H_{ij}(k)(x_i(k)-x_j(k))^2\cr
&\geq V(x(0))-\gamma
\sum_{k=0}^{t_1-1}\sum_{i<j}\left(A_{ij}(k)+A_{ji}(k)\right)(x_i(k)-x_j(k))^2.
\end{align}
In the last inequality we use the following relation:
\[H_{ij}(k)=\sum_{\ell=1}^mA_{\ell i}(k)A_{\ell j}(k)
\geq A_{ii}(k)A_{ij}(k)+A_{jj}(k)A_{ji}(k)\geq \gamma(A_{ij}(k)+A_{ji}(k)),\]
where the last inequality holds since $\Ac$ has the trivial permutation
component.
Therefore, using Eq.~\eqref{eqn:rate3} and Eq.~\eqref{eqn:rate4}, we conclude
that
\[V(x(t_1))\leq V(x(0))-\frac{\gamma\delta(1-\delta)^2}{m(m-1)^2}\, V(x(0))
=\left(1-\frac{\gamma\delta(1-\delta)^2}{m(m-1)^2}\right)V(x(0)).\]
\end{proof}

\bibliographystyle{IEEEtran}
\bibliography{behrouz-prelimV2}
\end{document}